\documentclass[12pt,reqno]{article}

\usepackage[usenames]{color}
\usepackage{amssymb}
\usepackage{amsmath}
\usepackage{amsthm} 
\usepackage{amsfonts}
\usepackage{amscd}
\usepackage{graphicx}

\usepackage[colorlinks=true, 
linkcolor=webgreen,
filecolor=webbrown, 
citecolor=webgreen]{hyperref}

\definecolor{webgreen}{rgb}{0,.5,0}
\definecolor{webbrown}{rgb}{.6,0,0}

\usepackage{color} 
\usepackage{fullpage} 
\usepackage{float}

\DeclareMathOperator{\rad}{rad}
\newcommand{\seqnum}[1]{\href{https://oeis.org/#1}{\rm\underline{#1}}}

\begin{document}
	
	\begin{center} 
	\end{center} 
	
	\theoremstyle{plain} 
	\newtheorem{theorem}{Theorem}
	\newtheorem{corollary}[theorem]{Corollary}
	\newtheorem{lemma}[theorem]{Lemma}
	\newtheorem{proposition}[theorem]{Proposition}
	
	\theoremstyle{definition}
	\newtheorem{definition}[theorem]{Definition}
	\newtheorem{example}[theorem]{Example}
	\newtheorem{conjecture}[theorem]{Conjecture}
	
	\theoremstyle{remark}
	\newtheorem{remark}[theorem]{Remark}
	
	\begin{center} \vskip 1cm{\LARGE\bf 
			On the Ternary Purely Exponential Diophantine  Equation $(ak)^x+(bk)^y=((a+b)k)^z$ with Prime Powers $a$ and $b$}
		
		\vskip 1cm 
		\large 
		{Maohua Le}\\
		
		Institute of Mathematics\\ 
		Lingnan Normal College\\
		Zhanjiang, Guangdong 524048\\ 
		China\\
		\href{mailto: lemaohua2008@163.com}{\tt lemaohua2008@163.com}\\
		\vskip 1cm
		{G\"{o}khan Soydan}\\
		Department of Mathematics \\
		Bursa Uluda\u{g} University\\ 
		16059 Bursa\\T\"urkiye\\
		\href{mailto: gsoydan@uludag.edu.tr}{\tt gsoydan@uludag.edu.tr}\\ 
	\end{center}
	
	\vskip .2 in
	
	\begin{abstract} Let $k$ be a positive integer, and let $a,b$ be coprime positive integers with $\min\{a,b\}>1$. In this paper, using a combination of some elementary number theory techniques with classical results on the Nagell-Ljunggren equation, the Catalan equation and some new properties of the Lucas sequence (\seqnum{A000204} in OEIS), we prove that if $k>1$ and $a,b$ are both prime powers with $\min\{a,b\}>2$, then the equation $(ak)^x+(bk)^y=((a+b)k)^z$ has only one positive integer solution $(x,y,z)=(1,1,1)$. The above result partially proves that Conjecture 1 presented in (Acta Arith. 2018, 184 (1): 37-49) is true.
	\end{abstract}
	
	\section{Introduction} \label{sec1}
	Let $\mathbb{Z}$, $\mathbb{N}$, $\mathbb{P}$ be the sets of all integers, positive integers and odd primes, respectively. Let $A,B,C$ be fixed positive integers with $\min\{A,B,C\}>1$. In recent decades, the ternary purely exponential Diophantine equation
	\begin{equation} \label{eq.1.1}
		A^x+B^y=C^z, \quad x,y,z\in\mathbb{N}
	\end{equation}
	has yielded very rich results, but some important problems about it are far from being solved (see\ \cite{LSS}). In recent 20 years, many authors have considered equation \eqref{eq.1.1} when $A,B$ and $C$ are Fibonacci \seqnum{A00045} or Lucas \seqnum{A000204} or Pell \seqnum{A000129} numbers in OEIS\ \cite{OEIS} or when $A$ and $B$ are Fibonacci or Lucas or Pell numbers (see\ \cite{AHRT,BrLu,DdLu,KoLu,RFLT,Si,ZaTo}).
	
	Let $k$ be a positive integer, and let $a,b$ be coprime positive integers with $\min\{a,b\}>1$. In this paper, we discuss \eqref{eq.1.1} for $(A,B,C)=(ak,bk,(a+b)k)$. Then, \eqref{eq.1.1} can be written as
	\begin{equation}\label{eq.1.2}
		(ak)^x+(bk)^y=((a+b)k)^z,\quad x,y,z\in\mathbb{N}.
	\end{equation}
	Obviously, for any $a,b$ and $k$, \eqref{eq.1.2} has the solution $(x,y,z)=(1,1,1)$. A solution $(x,y,z)$ of \eqref{eq.1.2} with $(x,y,z)\neq (1,1,1)$ is called exceptional. In 2018, Sun and Tang\ \cite{SuTa} proved that if $k>1$, $\min\{a,b\}>2$ and $(x,y,z)$ is an exceptional solution of \eqref{eq.1.2}, then either $x>z>y$ or $y>z>x$. On this basis, they further proved that if $k>1$ and $(a,b)\in\{(3,5), (5,8),(8,13),(13,21)\}$, then \eqref{eq.1.2} has no exceptional solutions. In the same year, Yuan and Han\ \cite{YuHa} proposed the following conjecture.
	\begin{conjecture}\label{con}
		For any $k$, if $\min\{a,b\}>3$, then \eqref{eq.1.2} has no exceptional solutions.
	\end{conjecture}
	The above conjecture is formally similar to Je\'smanowicz' conjecture concerning Pythagorean triples (see\ \cite{Jes,SCDT}). So far, it has only been solved for some very special cases. For example, the authors of\ \cite{YuHa} proved that if $a$ and $b$ are squares with $b\equiv 4 \pmod{8}$, then \eqref{eq.1.2} has no solutions $(x,y,z)$ with $y>z>x$, in particular, if $a$ is a square and $b=4$, then Conjecture \ref{con} is true. Afterwards, using Baker's method, Le and Soydan\ \cite{LS} proved that if $a$ and $b$ are squares with $a>64b^3$, then \eqref{eq.1.2} has no solutions $(x,y,z)$ with $x>z>y$. It implies that if $a$ and $b$ are squares with $a>64b^3$ and $b\equiv 4 \pmod{8}$, then Conjecture \ref{con} is true.
	
	In this paper, using a combination of some elementary number theory techniques with classical results on the Nagell-Ljunggren equation, the Catalan equation and some new properties of the Lucas sequence \seqnum{A000204}, we prove the following result. 
	\begin{theorem}\label{theo.1.2}
		Let $r,s$ be positive integers, and let $p,q$ be distinct odd primes. If $k>1$ and $a,b$ satisfy one of the following conditions:
		\begin{enumerate}
			\item[\rm (i)] $(a,b)=(2^r,p^s)$ with $r>1$; or
			\item[\rm (ii)] $(a,b)=(p^r,2^s)$ with $s>1$; or
			\item[\rm (iii)] $(a,b)=(p^r,q^s)$,
			then \eqref{eq.1.2} has no solutions $(x,y,z)$ with $x>z>y$.
		\end{enumerate}
	\end{theorem}   
	Since $a$ and $b$ are symmetric in \eqref{eq.1.2}, by the first result mentioned in\ \cite{SuTa}, we can obtain the following corollary from Theorem \ref{theo.1.2} immediately.
	\begin{corollary}\label{cor.1.3}
		If $k>1$ and $a,b$ are prime powers, then Conjecture \ref{con} is true. 
	\end{corollary}
	Finally, we briefly analyze the effect of the above mentioned results. For any enough large positive integer $N$, let $F(N)$ denote the number of pairs $(a,b)$ such that $\max\{a,b\}\le N$ and they have been proved to hold for Conjecture \ref{con} with $k>1$. Clearly, by the mentioned results in\ \cite{LS,SuTa,YuHa} and Corollary \ref{cor.1.3}, we have $F(N)=4$, $F(N)=\sqrt{N}$, $N>F(N)>\sqrt{N}$ and $F(N)>N^2/(\log N)^2$, respectively. 
	\section{Preliminaries} \label{sec2}
	Let us now recall that if $\alpha$ and $\beta$ are roots of
	a quadratic equation of the form $x^2-rx-s=0$ with nonzero coprime integers $%
	r$ and $s$ and such that $\alpha/\beta$ is not a root of
	unity, then the sequence $(u_{\ell})_{\ell\ge 0}$ of general term 
	\begin{equation*}
		u_{\ell}=\dfrac{\alpha^{\ell}-{\beta}^{\ell}}{\alpha-{\beta}}%
		\qquad {\text{for~all}}~\ell\ge 0 
	\end{equation*}
	is called a \textit{Lucas sequence} \seqnum{A000204}. It can also be defined inductively as $%
	u_0=0,~u_1=1$ and $u_{\ell+2}=r\cdot u_{\ell+1}+s\cdot u_{\ell}$. When the case $\beta=1$, the following Lemmas \ref{lem.2.6}, \ref{lem.2.7} and \ref{lem.2.8} are the three new properties we proved about the Lucas sequence at the beginning of page 4.
	\begin{lemma}[\cite{Lj}]\label{lem.2.1}
		The equation
		\begin{equation*}
			\dfrac{X^m-1}{X-1}=Y^n,\,\,\, X,Y,m,n\in\mathbb{N},\,\,\, X>1,\,\,Y>1,\,\,m>2,\,\,n>1
		\end{equation*}
		has only the solution $(X,Y,m,n)=(3,11,5,2)$ with $2\mid n.$
	\end{lemma}
	
	\begin{lemma}[\cite{Mih}]\label{lem.2.2}
		The equation
		\begin{equation*}
			X^m-Y^n=1,\,\,\, X,Y,m,n\in\mathbb{N},\,\,\,\min\{X,Y,m,n\}>1
		\end{equation*}
		has only the solution $(X,Y,m,n)=(3,2,2,3)$.
	\end{lemma}
	
	Let $X,\ell,m,n$ be positive integers with $X>1$ and $\ell>1$, and let $p,q$ be odd primes. Three lemmas for divisibility are given directly below.
	
	\begin{lemma}[\cite{Nag}]\label{lem.2.3}
		If $X^n+1\equiv 0 \pmod{X^m+1}$, then $m\mid n$ and $n/m$ is odd.
	\end{lemma}
	
	\begin{lemma}[\cite{BV}]\label{lem.2.4}
		Let $\ell$ be an odd prime. If $X\not \equiv 1 \pmod{\ell}$, then every prime divisor $p$ of $(X^{\ell}-1)/(X-1)$ satisfies
		\begin{equation} \label{eq.2.1}
			p\equiv 1 \pmod{2\ell}.
		\end{equation}
		If $X \equiv 1 \pmod{\ell}$, then $\ell \mid \mid (X^{\ell}-1)/(X-1)$ and every prime divisor $p$ of $(X^{\ell}-1)/\ell (X-1)$ satisfies \eqref{eq.2.1}.  
	\end{lemma}
	
	\begin{lemma}[\cite{BV}]\label{lem.2.5}
		When $X \equiv 1 \pmod{p}$, $p^m \mid \mid (X^{\ell}-1)/(X-1)$ if and only if $p^m \mid \mid  \ell$.   
	\end{lemma}
	
	\begin{lemma}\label{lem.2.6}
		If
		\begin{equation}\label{eq.2.2}
			\dfrac{X^{\ell}-1}{X-1}=p^n,\,\,2 \nmid \ell, 
		\end{equation}	
		then $\ell$ is an odd prime with \eqref{eq.2.1}. 
	\end{lemma}
	\begin{proof}
		We now assume that $\ell$ is not an odd prime. Since $\ell>1$ and $2\nmid \ell$, $\ell$ has an odd prime divisor $q$ with $q<\ell$. Then we have $q\mid \ell$ and $\ell/q>1$. By \eqref{eq.2.2}, we get
		\begin{equation}\label{eq.2.3}
			p^n=\dfrac{X^{\ell}-1}{X-1}=\left(\dfrac{X^{\ell/q}-1}{X-1}\right)\left(\dfrac{(X^{\ell/q})^q-1}{X^{\ell/q}-1}\right),
		\end{equation} 
		where $(X^{\ell/q}-1)/(X-1)$ and $((X^{\ell/q})^q-1)/(X^{\ell/q}-1)$ are positive integers greater than 1. Since $X>1$, by \eqref{eq.2.3}, we have
		\begin{equation}\label{eq.2.4}	\dfrac{X^{\ell/q}-1}{X-1}=p^f,\,\,\dfrac{(X^{\ell/q})^q-1}{X^{\ell/q}-1}=p^{n-f},\,\,f\in\mathbb{N},\,\,f<n.
		\end{equation}
		Further, by the first and the second equalities of \eqref{eq.2.4}, we get $X^{\ell/q}\equiv 1 \pmod{p}$ and 
		\begin{equation}\label{eq.2.5}
			0\equiv p^{n-f}\equiv \dfrac{(X^{\ell/q})^q-1}{X^{\ell/q}-1}\equiv (X^{\ell/q})^{q-1}+\cdots +X^{\ell/q}+1\equiv q \pmod{p}.
		\end{equation}
		Since $p$ and $q$ are odd primes, by \eqref{eq.2.5}, we obtain $q=p$. Hence by Lemma \ref{lem.2.4}, we see from \eqref{eq.2.4} that $p\mid \mid ((X^{\ell/p})^p-1)/(X^{\ell/p}-1)=((X^{\ell/q})^q-1)/(X^{\ell/q}-1)$ and 
		\begin{equation}\label{eq.2.6}
			p= \dfrac{(X^{\ell/p})^p-1}{X^{\ell/p}-1}= (X^{\ell/p})^{p-1}+\cdots X^{\ell/p}+1>p,
		\end{equation}
		a contradiction. It implies that $\ell$ must be an odd prime. Moreover, using Lemma \ref{lem.2.4} again, if $X\equiv 1 \pmod{\ell}$, then from \eqref{eq.2.2} we can get $\ell=p$ and $n=1$, which is the same contradiction as \eqref{eq.2.6}. Therefore, we have $x\not \equiv 1 \pmod{\ell}$ and $p$ satisfies \eqref{eq.2.1}. The lemma is proved.
	\end{proof}
	\begin{lemma}\label{lem.2.7}
		If
		\begin{equation}\label{eq.2.7}
			X^{\ell}-1=2^mp^n,
		\end{equation}
		then one of the following three conclusions holds.
		\begin{enumerate}
			\item[\rm (i)] $(p,X,\ell,m,n)=(3,5,2,3,1),(3,7,2,4,1),(5,9,2,4,1),(5,3,4,4,1),(3,17,2,5,2)$ or\\ $(7,15,2,5,1)$.
			\item[\rm (ii)]\begin{equation}\label{eq.2.8}
				\ell=2,\,\, m\ge 6,\,\, n=1,\,\, X=2^{m-1}+\zeta,\,\, p=2^{m-2}-\zeta,\,\, \zeta\in\{1,-1\}.
			\end{equation}
			\item[\rm (iii)] 
			\begin{equation}\label{eq.2.9}
				\ell\,\, \text{is an odd prime},\,\,\, X-1=2^m,\,\, \dfrac{X^{\ell}-1}{X-1}=p^n,\,\, p\equiv 1 \pmod{2\ell}.
			\end{equation}
		\end{enumerate}
	\end{lemma}	
	\begin{proof}
		Obviously, by \eqref{eq.2.7}, we have $X>1$, $2\nmid X$ and $p\nmid X$. When $2\mid \ell$, since $X^{\ell}\equiv 1 \pmod{8}$, $\gcd(X^{\ell/2}+1,X^{\ell/2}-1)=2$ and $p$ is an odd prime, by \eqref{eq.2.7}, we have
		\begin{equation}\label{eq.2.10}
			X^{\ell/2}+\zeta=2p^n,\,\,X^{\ell/2}-\zeta=2^{m-1},\,\,m\ge 3,\,\,\zeta=((-1)^{(X-1)/2})^{\ell/2},
		\end{equation}
		where
		\begin{equation}\label{eq.2.11}
			\zeta\in\{1,-1\}.
		\end{equation}
		If $3\le m\le 5$, then from \eqref{eq.2.10} and \eqref{eq.2.11} we can easily obtain the conclusion \rm (i). If $m\ge 6$, by Lemma \ref{lem.2.2}, then from the second equality of \eqref{eq.2.10} we get $\ell=2$. In addition, eliminating $X^{\ell/2}$ from \eqref{eq.2.10}, we have
		\begin{equation}\label{eq.2.12}
			p^n-2^{m-2}=\zeta.	
		\end{equation}
		Since $m-2\ge 4$, using Lemma \ref{lem.2.2} again, we see from \eqref{eq.2.12} that $n=1$. Therefore, by \eqref{eq.2.10}, \eqref{eq.2.11} and \eqref{eq.2.12}, we obtain \eqref{eq.2.8} and conclusion \rm (ii) is proved. 
		
		When $2\nmid \ell$, since $\ell>1$ and $\dfrac{X^{\ell}-1}{X-1}$ is an odd positive integer greater than 1, by \eqref{eq.2.7}, we have
		\begin{equation}\label{eq.2.13}
			X-1=2^mp^f,\,\,\dfrac{X^{\ell}-1}{X-1}=p^{n-f},\,\,f\in\mathbb{Z},\,\,0\le f<n.
		\end{equation}
		If $f>0$, then from the first equality of \eqref{eq.2.13} we get $X\equiv 1 \pmod{p}$. Hence, applying Lemma \ref{lem.2.5} to the second equality of \eqref{eq.2.13}, we obtain $p^{n-f}\mid \mid \ell.$ However, since $\ell\ge p^{n-f},$ we get $p^{n-f}=(X^{\ell}-1)/(X-1)>\ell \ge p^{n-f},$ a contradiction. So, we have $f=0$. Therefore, by \eqref{eq.2.13}, we get
		\begin{equation}\label{eq.2.14}
			X-1=2^m,\,\,\dfrac{X^{\ell}-1}{X-1}=p^n.
		\end{equation}
		Further, by Lemma \ref{lem.2.6}, we see from the second equality of \eqref{eq.2.14} that $\ell$ is an odd prime with \eqref{eq.2.1}. Thus, we obtain \eqref{eq.2.9} and the conclusion \rm (iii) is proved. The proof is completed.
	\end{proof}
	Using the same method as in the proof of Lemma \ref{lem.2.7}, we can obtain the following lemma without difficulty
	\begin{lemma}\label{lem.2.8}
		If
		\begin{equation}\label{eq.2.15}
			X^{\ell}-1=p^mq^n,
		\end{equation}
		the one of the following six conclusions holds.
		\begin{enumerate}
			\item[\rm (i)]
			\begin{equation}\label{eq.2.16}
				2\mid \ell,\,\,X^{\ell/2}+\zeta=p^m,\,\,X^{\ell/2}-\zeta=q^n,\,\,\zeta\in\{1,-1\}.
			\end{equation}
			\item[\rm (ii)]
			\begin{equation}\label{eq.2.17}
				2\nmid \ell,\,\,X=2,\,\,2^{\ell}-1=p^mq^n.
			\end{equation}
			\item[\rm (iii)]
			\begin{equation}\label{eq.2.18}
				\ell=p,\,\,m>1,\,\,X-1=p^{m-1},\,\,\dfrac{X^p-1}{X-1}=pq^n,\,\,q\equiv 1 \pmod{2p}.
			\end{equation}
			\item[\rm (iv)]
			\begin{equation}\label{eq.2.19}
				\ell=q,\,\,n>1,\,\,X-1=q^{n-1},\,\,\dfrac{X^q-1}{X-1}=p^mq,\,\,p\equiv 1 \pmod{2q}.
			\end{equation}
			\item[\rm (v)]
			\begin{equation}\label{eq.2.20}
				\ell \,\,\text{is an odd prime},\,\,X-1=p^m,\,\,\dfrac{X^{\ell}-1}{X-1}=q^n,\,\,q\equiv 1 \pmod{2\ell}.
			\end{equation}
			\item[\rm (vi)]
			\begin{equation}\label{eq.2.21}
				\ell \,\,\text{is an odd prime},\,\,X-1=q^n,\,\,\dfrac{X^{\ell}-1}{X-1}=p^m,\,\,p\equiv 1 \pmod{2\ell}.
			\end{equation}
		\end{enumerate}
	\end{lemma} 
	
	\begin{proof}
		Since $p$ and $q$ are distinct odd primes, by \eqref{eq.2.15}, we have $2\mid X$. When $2\mid \ell$, since $\gcd(X^{\ell/2}+1,X^{\ell/2}-1)=1$, by \eqref{eq.2.15}, we can directly obtain \eqref{eq.2.16} and the conclusion \rm (i) is proved. When $2\nmid \ell$ and $X=2$, we see from \eqref{eq.2.15} that \eqref{eq.2.17} is clearly true and the conclusion \rm (ii) is proved. When $2\nmid \ell$ and $X>2$, by \eqref{eq.2.15}, we have
		\begin{equation}\label{eq.2.22}
			\begin{aligned}
				&X-1=p^fq^g,\,\,\dfrac{X^{\ell}-1}{X-1}=p^{m-f}q^{n-g},\,\,f,g\in\mathbb{Z},\,\,0\le f\le m,\\
				&0\le g\le n,\,\,
				(f,g)\neq (0,0)\,\,\text{or}\,\,(m,n).
			\end{aligned}
		\end{equation}
		
		If $0<f<m$ and $0<g<n$, then from \eqref{eq.2.22} we get $X\equiv 1 \pmod{pq}$ and $0\equiv p^{m-f}q^{n-g}\equiv (X^{\ell}-1)(X-1)\equiv X^{\ell-1}+\cdots +X+1\equiv \ell \pmod{pq}$. It follows that $pq\mid \ell$ and $(X^{pq}-1)/(X-1)\mid (X^{\ell}-1)(X-1).$ Hence, by the second equality of \eqref{eq.2.22}, we have
		\begin{equation}\label{eq.2.23}
			\begin{aligned}
				&\dfrac{X^{pq}-1}{X-1}=\left(\dfrac{X^{p}-1}{X-1}\right)\left(\dfrac{(X^p)^q-1}{X^p-1}\right)=p^{f'}q^{g'},\,\,f',g'\in\mathbb{Z},\\
				&f'\ge 0,\,\,g'\ge 0,\,\,(f',g')\neq (0,0), 
			\end{aligned}
		\end{equation}
		where $(X^p-1)/(X-1)$ and $((X^p)^q-1)/(X^p-1)$ are positive integers greater than 1. Further, since $X\equiv 1 \pmod{p}$, by Lemma \ref{lem.2.5}, we get $p\mid \mid (X^p-1)/(X-1)$. Furthermore, since $(X^p-1)/(X-1)>p,$ by \eqref{eq.2.23}, we have
		\begin{equation}\label{eq.2.24}
			\dfrac{X^p-1}{X-1}=pq^{g''},\,\,g''\in \mathbb{N}.
		\end{equation}
		Recall that $X\equiv 1 \pmod{q}$. By \eqref{eq.2.24}, we get $0\equiv pq^{g''}\equiv (X^p-1)/(X-1)\equiv X^{p-1}+\cdots +X+1\equiv p \pmod{q}$ and $p=q$, a contradiction. Therefore, the case $0<f<m$ and $0<g<n$ is impossible.
		
		If $0<f<m$ and $g=n$, then from \eqref{eq.2.22} we get
		\begin{equation}\label{eq.2.25}
			X-1=p^fq^n,\,\,\dfrac{X^{\ell}-1}{X-1}=p^{m-f}.
		\end{equation}   
		By the first equality of \eqref{eq.2.25}, we have $X\equiv 1 \pmod{p}$. In addition, by Lemma \ref{eq.2.6}, we see from the second equality of \eqref{eq.2.25} that $\ell$ is an odd prime with $p\equiv 1 \pmod{2 \ell}$. But, since $X\equiv 1 \pmod{p}$, we get $0\equiv p^{m-f}\equiv (X^{\ell}-1)/(X-1)\equiv X^{\ell-1}+\cdots +X+1\equiv \ell \pmod{p}$ and $p=\ell$, a contradiction. Therefore, the case $0<f<m$ and $g=n$ is impossible. Using the same method, we can eliminate the case $f=m$ and $0<g<n$.
		
		Recall that $(f,g)\neq (0,0)$ or $(m,n)$. According to the above analysis, we are left with only the cases
		\begin{equation}\label{eq.2.26}
			f=0,\,\,0<g\le n
		\end{equation}
		and
		\begin{equation}\label{eq.2.27}
			0<f\le m,\,\,g=0
		\end{equation}
		that have not yet been discussed.
		
		If \eqref{eq.2.26} holds, by \eqref{eq.2.22}, then we have
		\begin{equation}\label{eq.2.28}
			X-1=q^g,\,\,\dfrac{X^{\ell}-1}{X-1}=p^mq^{n-g},\,\,g\in\mathbb{N},\,\,g\le n.
		\end{equation}
		When $g=n$, by \eqref{eq.2.28}, we get
		\begin{equation}\label{eq.2.29}
			X-1=q^n,\,\,\dfrac{X^{\ell}-1}{X-1}=p^m.
		\end{equation}
		Applying Lemma \ref{lem.2.6} to the second equality of \eqref{eq.2.29}, $\ell$ is an odd prime with $p\equiv 1 \pmod{2\ell}$. Hence, by \eqref{eq.2.29}, we get \eqref{eq.2.21} and the conclusion \rm(vi) is proved.
		
		When $g<n$, by the first equality of \eqref{eq.2.28}, we have $X\equiv 1\pmod{q}$. Hence, by Lemma \ref{lem.2.5}, we see from the second equality of \eqref{eq.2.28} that $q^{n-g}\mid \mid \ell$. It follows that
		\begin{equation}\label{eq.2.30}
			\ell=q^{n-g}\ell_1,\,\,\ell_1\in\mathbb{N},\,\,q\nmid \ell_1.
		\end{equation}
		Then we have
		\begin{equation}\label{eq.2.31}
			\dfrac{X^{\ell}-1}{X-1}=\left(\dfrac{X^{\ell_1}-1}{X-1}\right)\prod_{j=1}^{n-g}\left(\dfrac{X^{q^j\ell_1}-1}{X^{q^{j-1}\ell_1}-1}\right),
		\end{equation}
		where $(X^{\ell_1}-1)/(X-1)$ and $(X^{q^j\ell_1}-1)/(X^{q^{j-1}\ell_1}-1)$ $(j=1,\cdots ,n-g)$ are positive integers with 
		\begin{equation}\label{eq.2.32}
			q\nmid \dfrac{X^{\ell_1}-1}{X-1},\,\,q \mid \mid \dfrac{X^{q^j\ell_1}-1}{X^{q^{j-1}\ell_1}-1},\,\,j=1,\cdots , n-g.
		\end{equation}
		By \eqref{eq.2.31} and \eqref{eq.2.32}, we get from the second equality of \eqref{eq.2.28} that
		\begin{equation}\label{eq.2.33}
			\dfrac{X^{\ell_1}-1}{X-1}=p^{m_0},\,\,m_0\in\mathbb{Z},\,\,m_0\ge 0
		\end{equation}
		and
		\begin{equation}\label{eq.2.34}
			\dfrac{X^{q^j\ell_1}-1}{X^{q^{j-1}\ell_1}-1}=p^{m_j}q,\,\,m_j\in\mathbb{N},\,\,j=1,\cdots,n-g,
		\end{equation}
		where
		\begin{equation}\label{eq.2.35}
			m_0+m_1+\cdots +m_{n-g}=m.
		\end{equation}
		If $n-g>1$, then from \eqref{eq.2.34} we have $X^{q\ell_1}\equiv 1 \pmod{p}$ and $0\equiv p^{m_2}q\equiv (X^{q^2\ell_1}-1)/(X^{q\ell_1}-1)\equiv X^{q\ell_1(q-1)}+\cdots +X^{q\ell_1}+1\equiv q \pmod{p}$, whence we get $p=q$, a contradiction. So we have $n-g=1$. Further, if $n-g=1$ and $\ell_1>1$, by \eqref{eq.2.33} and \eqref{eq.2.34}, then $m_0$ is a positive integer, $X^{\ell_1}\equiv 1 \pmod{p}$, $0\equiv p^{m_1}q\equiv (X^{q\ell_1}-1)/(X^{\ell_1}-1)\equiv X^{\ell_1(q-1)}+\cdots +X^{\ell_1}+1\equiv q \pmod{p}$ and $p=q$, a contradiction. Therefore, if \eqref{eq.2.26} holds, then we have $n-g=1$ and $\ell_1=1.$ By \eqref{eq.2.28},\eqref{eq.2.30},\eqref{eq.2.31}, \eqref{eq.2.34} and \eqref{eq.2.35}, we get \eqref{eq.2.19} and the conclusion \rm (iv) is proved.
		
		Using the same method, as in the proof about the case \eqref{eq.2.26}, we can deduce that if \eqref{eq.2.27} holds, then there can only obtain \eqref{eq.2.18} or \eqref{eq.2.20}, and the conclusions \rm (iii) and \rm (v) are proved. To sum up, the proof is complete.
	\end{proof}
	For any positive integer $m$, let $\rad(m)$ denote the product of all distinct prime divisors of $m$, and let $\rad(1)=1.$ Obviously, $\rad(m)$ is equal to the largest squarefree divisor of $m$.
	\begin{lemma}[\cite{SuTa}]\label{lem.2.9}
		If $k>1$ and $(x,y,z)$ is a solution of \eqref{eq.1.2} with $x>z>y$, then we have
		$$\rad(k)\mid b,\,\,b=b_1b_2,\,\,b_1,b_2\in\mathbb{N},\,\,b_1>1,\,\,\gcd(b_1,b_2)=1,$$
		$$b_1^y=k^{z-y}$$
		and
		$$a^xk^{x-z}+b_2^y=(a+b)^z.$$ 
	\end{lemma}
	By Lemma \ref{lem.2.9}, we can obtain the following lemma immediately.
	\begin{lemma}\label{lem.2.10}
		If $k>1$, $b$ is a prime power and $(x,y,z)$ is a solution of \eqref{eq.1.2} with $x>z>y$, then we have
		$$b^y=k^{z-y}$$
		and
		$$a^xk^{x-z}+1=(a+b)^z.$$
	\end{lemma}
	\section{Proof of Theorem \ref{theo.1.2}}
	We now assume that $k>1$ and $(x,y,z)$ is a solution of \eqref{eq.1.2} with $x>z>y$. We will prove that this solution does not exist in the following three cases.
	\medskip\noindent
	\subsection{Case $(i)$ : $(a,b)=(2^r,p^s)$ with $r>1$.}
	Since $k>1$ and $p$ is an odd prime, Lemma \ref{lem.2.10}, we have
	\begin{equation}\label{eq.3.1}
		p^{sy}=k^{z-y}
	\end{equation}
	and 
	\begin{equation}\label{eq.3.2}
		2^{rx}k^{x-z}+1=(2^r+p^s)^z.
	\end{equation}
	We see from \eqref{eq.3.1} that $k$ is a power of $p$. So we have
	\begin{equation}\label{eq.3.3}
		k^{x-z}=p^t,\,\,t\in\mathbb{N}.
	\end{equation}
	Substituting \eqref{eq.3.3} into \eqref{eq.3.2}, we get
	\begin{equation}\label{eq.3.4}
		(2^r+p^s)^z-1=2^{rx}p^t.
	\end{equation}
	We find from \eqref{eq.3.4} that equation
	\begin{equation}\label{eq.3.5}
		X^{\ell}-1=2^mp^n,\,\,X,\ell,m,n\in\mathbb{N},\,\,X>1,\ell>1
	\end{equation}
	has a solution
	\begin{equation}\label{eq.3.6}
		(X,\ell,m,n)=(2^r+p^s,z,rx,t).
	\end{equation}
	Since $r\ge 2$ and $x>z>y\ge 1$, we have $x\ge 3$ and $rx\ge 6.$ Hence, by Lemma \ref{lem.2.7}, the solution \eqref{eq.3.6} must satisfy the conclusion $(ii)$ or $(iii)$ in this lemma.
	
	When \eqref{eq.3.6} satisfies the conclusion $(ii)$ in Lemma \ref{lem.2.7}, by \eqref{eq.2.8} and \eqref{eq.3.6}, we have
	\begin{equation}\label{eq.3.7}
		z=2,\,\,t=1,\,\,2^r+p^s=2^{rx-1}+\zeta,\,\,p=2^{rx-2}+\zeta,\,\,\zeta\in\{1,-1\}.
	\end{equation}
	Since $z=2$ and $z>y\ge 1$, we get $y=1$. Hence, by \eqref{eq.3.1}, we have
	\begin{equation}\label{eq.3.8}
		k=p^s.
	\end{equation}
	Further, by \eqref{eq.3.3}, \eqref{eq.3.6} and \eqref{eq.3.8}, we get $p=p^t=k^{x-z}=p^{s(x-z)}=p^{s(x-2)}$, whence we obtain $s=1$ and $x=3$. Therefore, by the third and fourth equalities of \eqref{eq.3.7}, we have
	$$2^{3r-1}+\zeta=2^{rx-1}+\zeta=2^r+p^s=2^r+p=2^r+(2^{rx-2}+\zeta)=2^{3r-2}+2^r+\zeta,$$
	whence we get $2^r=2^{3r-1}-2^{3r-2}=2^{3r-2}$ and $r=1$, a contradiction.
	
	When \eqref{eq.3.6} satisfies the conclusion $(iii)$ in Lemma \ref{lem.2.7}, by \eqref{eq.2.9} and \eqref{eq.3.6}, we have
	\begin{equation}\label{eq.3.9}
		z\,\, \text{is an odd prime},\,\,2^r+p^s-1=2^{rx},\,\, \dfrac{(2^r+p^s)^z-1}{2^r+p^s-1}=p^t,\,\, p\equiv 1 \pmod{2z}.
	\end{equation}
	By the second equality of \eqref{eq.3.9}, we get
	\begin{equation}\label{eq.3.10}
		p^s=2^{rx}-2^r+1.
	\end{equation}
	If $t\ge s$, then from the third equality of \eqref{eq.3.9} we have
	\begin{equation*}
		0\equiv p^t\equiv\dfrac{(2^r+p^s)^z-1}{2^r+p^s-1}=\dfrac{2^{rz}-1}{2^r-1}\,\,\pmod{p^s}.
	\end{equation*}
	It implies that $(2^{rz}-1)/(2^r-1)$ is a positive integer satisfying
	\begin{equation}\label{eq.3.11}
		\dfrac{2^{rz}-1}{2^r-1}\ge p^s.
	\end{equation}
	Since $x>z$, by \eqref{eq.3.10} and \eqref{eq.3.11}, we have
	\begin{equation*}
		2^{r(z-1)+1}>\dfrac{2^{rz}-1}{2^r-1}\ge p^s=2^{rx}-2^r+1\ge 2^{r(z+1)}-2^r+1>2^{r(z+1)}-2^r
	\end{equation*}
	and $r\ge 2r-1$, a contradiction. So we have $t<s$. Then, since $z>2$, by the third equality of \eqref{eq.3.9}, we get
	$$p^s>p^t=\dfrac{(2^r+p^s)^z-1}{2^r+p^s-1}>\dfrac{(2^r+p^s)^2-1}{2^r+p^s-1}=2^r+p^s+1>p^s,$$
	a contradiction. Thus, the theorem holds for this case.
	
	\medskip\noindent
	\subsection{Case $(ii)$ : $(a,b)=(p^r,2^s)$ with $s>1$.}
	By Lemma \ref{lem.2.10}, we have 
	\begin{equation}\label{eq.3.12}
		2^{sy}=k^{z-y}
	\end{equation}
	and
	\begin{equation}\label{eq.3.13}
		p^{rx}k^{x-z}+1=(p^r+2^s)^z.
	\end{equation}
	We see from \eqref{eq.3.12} that $k$ is a power of 2. So we have
	\begin{equation}\label{eq.3.14}
		k^{x-z}=2^t,\,\,t\in\mathbb{N}.
	\end{equation}
	Substituting \eqref{eq.3.14} into \eqref{eq.3.13}, we get
	\begin{equation}\label{eq.3.15}
		(p^r+2^s)^z-1=2^tp^{rx}.
	\end{equation}
	We find from \eqref{eq.3.15} that \eqref{eq.3.5} has a solution
	\begin{equation}\label{eq.3.16}
		(X,\ell,m,n)=(p^r+2^s,z,t,rx).
	\end{equation}
	Since $rx\ge x\ge 3,$ by Lemma \ref{lem.2.7}, the solution \eqref{eq.3.16} only satisfies the conclusion $(iii)$ in this lemma. Then, by \eqref{eq.2.9} and \eqref{eq.3.16}, we have
	\begin{equation}\label{eq.3.17}
		z\,\, \text{is an odd prime},\,\,p^r+2^s-1=2^t,\,\, \dfrac{(p^r+2^s)^z-1}{p^r+2^s-1}=p^{rx},\,\, p\equiv 1 \pmod{2z}.
	\end{equation}
	Further, since $rx\ge x>2$, applying Lemma \ref{lem.2.1} to the third equality of \eqref{eq.3.17}, we get
	\begin{equation}\label{eq.3.18}
		2\nmid x.
	\end{equation}
	
	By the first equality of \eqref{eq.3.17}, $z$ is an odd prime. Since $z>y\ge 1$, we have $\gcd(y,z)=1$ and
	\begin{equation}\label{eq.3.19}
		\gcd(y,z-y)=1.
	\end{equation}
	We see from \eqref{eq.3.12} and \eqref{eq.3.19} that $s\equiv 0 \pmod{z-y}$,
	\begin{equation}\label{eq.3.20}
		s=(z-y)s_1,\,\,s_1\in\mathbb{N}
	\end{equation}
	and
	\begin{equation}\label{eq.3.21}
		k=2^{s_1y}.
	\end{equation}
	Further, by \eqref{eq.3.14} and \eqref{eq.3.21}, we have
	\begin{equation}\label{eq.3.22}
		t=s_1y(x-z),
	\end{equation}
	whence we get
	\begin{equation}\label{eq.3.23}
		y\mid t.
	\end{equation}
	
	By the second equality of \eqref{eq.3.17}, we have
	\begin{equation}\label{eq.3.24}
		p^r+2^s=2^t+1
	\end{equation}
	and
	\begin{equation}\label{eq.3.25}
		p^r\equiv -2^s \pmod{2^t+1}.
	\end{equation}
	Substituting \eqref{eq.3.24} into the third equality of \eqref{eq.3.17}, we get
	\begin{equation}\label{eq.3.26}
		p^{rx}= \dfrac{(2^t+1)^z-1}{(2^t+1)-1}=(2^t+1)^{z-1}+\cdots+(2^t+1)+1
	\end{equation}
	and
	\begin{equation}\label{eq.3.27}
		p^{rx}\equiv 1 \pmod{2^t+1}.
	\end{equation}
	Further, by \eqref{eq.3.18}, \eqref{eq.3.25} and \eqref{eq.3.27}, we have
	\begin{equation}\label{eq.3.28}
		2^{sx}+1\equiv 0 \pmod{2^t+1}.
	\end{equation}
	Applying Lemma \ref{lem.2.3} to \eqref{eq.3.28}, we get $sx\equiv 0 \pmod{t}$ and
	\begin{equation}\label{eq.3.29}
		sx=tt_1,\,\,t_1\in\mathbb{N},\,\,2\nmid t_1.
	\end{equation} 
	Hence, by \eqref{eq.3.20}, \eqref{eq.3.22} and \eqref{eq.3.29}, we have 
	\begin{equation}\label{eq.3.30}
		(z-y)x=y(x-z)t_1.
	\end{equation}
	Furthermore, by \eqref{eq.3.19} and \eqref{eq.3.30}, we get
	\begin{equation}\label{eq.3.31}
		y\mid x.
	\end{equation}
	Therefore, by \eqref{eq.3.23}, \eqref{eq.3.26} and \eqref{eq.3.31}, we have
	\begin{equation}\label{eq.3.32}
		1=(2^t+1)^z-2^tp^{rx}=(2^t+1)^z-(2^{t/y}p^{rx/y})^y,
	\end{equation}
	where $2^{t/y}p^{rx/y}$ is a positive integer greater than 6. Since $z\ge 3,$ by Lemma \ref{lem.2.2}, we see from \eqref{eq.3.32} that
	\begin{equation}\label{eq.3.33}
		y=1.
	\end{equation}
	Thus, by \eqref{eq.3.20} and \eqref{eq.3.33}, we get 
	\begin{equation}\label{eq.3.34}
		s\ge z-1.
	\end{equation}
	
	On the other hand, we see from \eqref{eq.3.24} that
	\begin{equation}\label{eq.3.35}
		t>s
	\end{equation}
	and
	\begin{equation}\label{eq.3.36}
		p^r\equiv 1 \pmod{2^s}.
	\end{equation}
	Further, by \eqref{eq.3.26}, \eqref{eq.3.35} and \eqref{eq.3.36}, we have
	\begin{equation}\label{eq.3.37}
		1\equiv	p^{rx}\equiv \dfrac{(2^t+1)^z-1}{(2^t+1)-1}=(2^t+1)^{z-1}+\cdots+(2^t+1)+1\equiv z \pmod{2^s}.
	\end{equation}
	Since $z>1$, by \eqref{eq.3.37}, we get
	\begin{equation}\label{eq.3.38}
		z\ge 2^s+1.
	\end{equation}
	Therefore, the combination of \eqref{eq.3.34} and \eqref{eq.3.38} yields $s\ge z-1\ge 2^s$, a contradiction. Thus, the theorem holds in this case.
	
	\medskip\noindent
	\subsection{Case $(iii)$ : $(a,b)=(p^r,q^s)$}
	By Lemma \ref{lem.2.10}, we have
	\begin{equation}\label{eq.3.39}
		q^{sy}=k^{z-y}
	\end{equation}
	and
	\begin{equation}\label{eq.3.40}
		p^{rx}k^{x-z}+1=(p^r+q^s)^z.
	\end{equation}
	We see from \eqref{eq.3.39} that $k$ is a power of $q$. So we have
	\begin{equation}\label{eq.3.41}
		k^{x-z}=q^t,\,\,t\in\mathbb{N}.
	\end{equation}
	Substituting \eqref{eq.3.41} into \eqref{eq.3.40}, we get
	\begin{equation}\label{eq.3.42}
		(p^r+q^s)^z-1=p^{rx}q^t.
	\end{equation}
	We find from \eqref{eq.3.42} that the equation
	\begin{equation}\label{eq.3.43}
		X^{\ell}-1=p^mq^n,\,\,X,\ell,m,n\in\mathbb{N},\,\,X>1,\ell>1
	\end{equation}
	has a solution
	\begin{equation}\label{eq.3.44}
		(X,\ell,m,n)=(p^r+q^s,z,rx,t).
	\end{equation}
	Since $p^r+q^s\ge 8$, applying Lemma \ref{lem.2.8} to \eqref{eq.3.43} and \eqref{eq.3.44}, we only need to consider the following five subcases.
	
	\bigskip
	\noindent{\it Subcase $(iii)-1$ }:
	\begin{equation}\label{eq.3.45}
		2\mid z,\,\,(p^r+q^s)^{z/2}+\zeta=p^{rx},\,\,(p^r+q^s)^{z/2}-\zeta=q^t,\,\,\zeta\in\{1,-1\}.  
	\end{equation}
	
	Since $rx\ge 3$, by Lemma \ref{lem.2.2}, we see from the second equality of \eqref{eq.3.45} that $z=2.$ So we have
	\begin{equation}\label{eq.3.46}
		p^r+q^s+\zeta=p^{rx},\,\,p^r+q^s-\zeta=q^t.
	\end{equation}
	By the first equality of \eqref{eq.3.46}, we get
	\begin{equation}\label{eq.3.47}
		q^s=p^{rx}-p^r-\zeta=p^r(p^{r(x-1)}-1)-\zeta\ge p^r(p^{2r}-1)-1>p^r.
	\end{equation}
	However, by the second equality of \eqref{eq.3.46}, we have $t>s$ and 
	\begin{equation*}
		p^r=q^t-q^s+\zeta=q^s(q^{t-s}-1)-1\ge q^s(q-1)-1>q^s,
	\end{equation*}
	which contradicts \eqref{eq.3.47}. Therefore, this subcase can be eliminated.
	
	\bigskip
	\noindent{\it Subcase $(iii)-2$ }:
	\begin{equation}\label{eq.3.48}
		z=p,\,\,\,\,p^r+q^s-1=p^{rx-1},\,\, \dfrac{(p^r+q^s)^p-1}{p^r+q^s-1}=pq^t,\,\, q\equiv 1 \pmod{2p}.
	\end{equation}
	If $t\ge s$, then from the third equality of \eqref{eq.3.48} we get
	\begin{equation}\label{eq.3.49}
		0\equiv pq^t\equiv \dfrac{(p^r+q^s)^p-1}{p^r+q^s-1}\equiv \dfrac{p^{rp}-1}{p^r-1} \pmod{q^s},
	\end{equation}
	where $(p^{rp}-1)/(p^r-1)$ is a positive integer. However, by \eqref{eq.3.48} and \eqref{eq.3.49}, we have
	\begin{align*}
		2p^{r(z-1)}=2p^{r(p-1)}&=p^{r(p-1)}\sum_{i=0}^{\infty}\dfrac{1}{2^i}>p^{r(p-1)}\sum_{j=0}^{p-1}\dfrac{1}{p^{r_j}}=\dfrac{p^{rp}-1}{p^r-1} \\
		&\ge q^s=p^{rx-1}-p^r+1\ge p^{r(z+1)-1}-p^r+1,
	\end{align*}
	whence we get
	\begin{align*}
		p^r>p^r-1&\ge p^{r(z+1)-1}-2p^{r(z-1)}=p^{r(z-1)}(p^{2r-1}-2)\\ 
		&\ge p^{r(z-1)}(p-2)\ge p^{r(z-1)}\ge p^r,
	\end{align*}
	a contradiction. So we have $t<s$. Then, since $p$ is an odd prime, by the third equality of \eqref{eq.3.48}, we get
	\begin{equation*}
		pq^s>pq^t=\dfrac{(p^r+q^s)^p-1}{p^r+q^s-1}\ge \dfrac{(p^r+q^s)^3-1}{p^r+q^s-1}>(p^r+q^s)^2>pq^s,
	\end{equation*}
	a contradiction. Thus, this subcase can be eliminated.
	
	\bigskip
	\noindent{\it Subcase $(iii)-3$ }:
	
	\begin{equation}\label{eq.3.50}
		z=q,\,\,p^r+q^s-1=q^{t-1},\,\, \dfrac{(p^r+q^s)^q-1}{p^r+q^s-1}=p^{rx}q,\,\, p\equiv 1 \pmod{2q}.
	\end{equation}
	
	By \eqref{eq.3.50}, we have $p^r+q^s=q^{t-1}+1$ and
	\begin{equation*}
		p^{rx}q=\dfrac{(p^r+q^s)^q-1}{p^r+q^s-1}=\dfrac{(q^{t-1}+1)^q-1}{q^{t-1}}=\sum_{i=1}^{q-1}\binom{q}{i}q^{(t-1)(i-1)},
	\end{equation*}
	whence we get
	\begin{equation} \label{eq.3.51}
		p^{rx}=1+\sum_{i=2}^{q-1}\binom{q}{i}q^{(t-1)(i-1)-1}.
	\end{equation}
	On the other hand, by the second equality of \eqref{eq.3.50}, we have
	\begin{equation}\label{eq.3.52}
		s<t-1
	\end{equation}
	and
	\begin{equation}\label{eq.3.53}
		p^r=q^s(q^{t-s-1}-1)+1.
	\end{equation}
	Hence, by \eqref{eq.3.53}, we get
	\begin{equation} \label{eq.3.54}
		p^{rx}=1+\sum_{j=1}^{x}\binom{x}{j}(q^s(q^{t-s-1}-1))^j.
	\end{equation}
	The combination of \eqref{eq.3.51} and \eqref{eq.3.54} yields
	\begin{equation} \label{eq.3.55}
		\sum_{i=2}^{q-1}\binom{q}{i}q^{(t-1)(i-1)-1}=\sum_{j=1}^{x}\binom{x}{j}(q^s(q^{t-s-1}-1))^j.
	\end{equation}
	Let $q^\alpha\mid \mid x$, where $\alpha$ is a nonnegative integer. Notice that
	\begin{equation} \label{eq.3.56}
		q^{t-1}\mid \mid 	\sum_{i=2}^{q-1}\binom{q}{i}q^{(t-1)(i-1)-1},\,\,q^{\alpha+s}\mid \mid \sum_{j=1}^{x}\binom{x}{j}(q^s(q^{t-s-1}-1))^j.
	\end{equation}
	By \eqref{eq.3.55} and \eqref{eq.3.56}, we have $t-1=\alpha+s$. Further, by \eqref{eq.3.52}, we get $\alpha=(t-1)-s>0.$ It implies that
	\begin{equation}\label{eq.3.57}
		q\mid x.
	\end{equation}
	
	Recall that $z>y$ and $z=q$ is an odd prime. Hence, $y$ and $z$ satisfy \eqref{eq.3.19}. By \eqref{eq.3.19} and \eqref{eq.3.39}, we have $s\equiv 0 \pmod{z-y}$ and \eqref{eq.3.20}. Further, by \eqref{eq.3.39} and \eqref{eq.3.41} we get $k=q^{s_1y}$ and
	\begin{equation}\label{eq.3.58}
		t=s_1y(x-z)=s_1y(x-q).	
	\end{equation}
	Furthermore, by \eqref{eq.3.57} and \eqref{eq.3.58}, we obtain
	\begin{equation}\label{eq.3.59}
		q\mid t.
	\end{equation}
	Therefore, by \eqref{eq.3.50}, \eqref{eq.3.57} and \eqref{eq.3.59}, we have
	\begin{align*}
		p^{rx/q}q^{t/q}\in\mathbb{N},\,\,1&=(p^r+q^s)^q-p^{rx}q^t=(p^r+q^s)^q-(p^{rx/q}q^{t/q})^q\\
		&=(p^r+q^s-p^{rx/q}q^{t/q})\sum_{i=0}^{q-1}(p^r+q^s)^{q-i-1}(p^{rx/q}q^{t/q})^i>1,
	\end{align*}
a contradiction. Thus, this subcase can be eliminated.

\bigskip
\noindent{\it Subcase $(iii)-4$ }:
\begin{equation}\label{eq.3.60} z\,\,\text{is an odd prime},\,\,p^r+q^s-1=p^{rx},\,\, \dfrac{(p^r+q^s)^z-1}{p^r+q^s-1}=q^t,\,\, q\equiv 1 \pmod{2z}.
\end{equation}

We can eliminate this subcase by using the same method as in the proof of Subcase $(iii)-2$. An outline of this process is given below. If $t\ge s,$ then from the second and third equalities of \eqref{eq.3.60} we have
\begin{equation*}
	0\equiv q^t\equiv \dfrac{(p^r+q^s)^z-1}{p^r+q^s-1}\equiv \dfrac{p^{rz}-1}{p^r-1} \pmod{q^s},
\end{equation*}
whence we get $(p^{rz}-1)/(p^r-1)\ge q^s.$ So we have
\begin{equation*}
	p^{r(z-1)+1}>\dfrac{p^{rz}-1}{p^r-1}\ge q^s=p^{rx}-p^r+1\ge p^{r(z+1)}-p^r+1>p^{r(z+1)-1},
\end{equation*}	 
a contradiction. It implies that $t<s$. Then, by \eqref{eq.3.60}, we get
\begin{equation*}
	q^s>q^t=\dfrac{(p^r+q^s)^z-1}{p^r+q^s-1}>p^r+q^s+1>q^s,
\end{equation*}
a contradiction. Thus, this subcase can be eliminated.

\bigskip
\noindent{\it Subcase $(iii)-5$ }:
\begin{equation}\label{eq.3.61}
	z\,\,\text{is an odd prime},\,\,p^r+q^s-1=q^t,\,\, \dfrac{(p^r+q^s)^z-1}{p^r+q^s-1}=p^{rx},\,\, p\equiv 1 \pmod{2z}.
\end{equation}

We see from the second equality of \eqref{eq.3.61} that $t>s$. So we have $p^r=q^s(q^{t-s}-1)+1$ and
\begin{equation}\label{eq.3.62}
	p^r\equiv 1 \pmod{q^s}.
\end{equation}
By the second and third equalities of \eqref{eq.3.61}, we get
\begin{equation}\label{eq.3.63}
	p^r+q^s=q^t+1
\end{equation}
and
\begin{equation}\label{eq.3.64}
	p^{rx}=\dfrac{(p^r+q^s)^z-1}{p^r+q^s-1}= \dfrac{(q^t+1)^z-1}{(q^t+1)-1}.
\end{equation}
Further, by \eqref{eq.3.64}, we have
\begin{equation}\label{eq.3.65}
	p^{rx}\equiv (q^t+1)^{z-1}+\cdots +(q^t+1)+1\equiv z\pmod{q^t}.
\end{equation}
Since $t>s$, by \eqref{eq.3.62} and \eqref{eq.3.65}, we get $z\equiv 1 \pmod{q^s}$. Furthermore, since $z>1$, we have
\begin{equation}\label{eq.3.66}
	z\ge q^s+1.
\end{equation}

On the other hand, since $z$ is an odd prime with $z>y\ge 1$, $y$ and $z$ satisfy \eqref{eq.3.19}. Hence, by \eqref{eq.3.19} and \eqref{eq.3.39}, $s$ satisfies \eqref{eq.3.20}, whence we can obtain \eqref{eq.3.23}.

Since $p^r+q^s>3,$ applying Lemma \ref{lem.2.1} to the third equality of \eqref{eq.3.61}, we have
\begin{equation}\label{eq.3.67}
	2\nmid x.
\end{equation}
By \eqref{eq.3.64}, we have
\begin{equation}\label{eq.3.68}
	p^{rx}\equiv \dfrac{(q^t+1)^z-1}{q^t+1}\equiv(q^t+1)^{z-1}+\cdots+(q^t+1)+1\equiv 1 \pmod{q^t+1}.
\end{equation}
Since $p^r\equiv -q^s \pmod{q^t+1}$ by \eqref{eq.3.63}, we get from \eqref{eq.3.67} and \eqref{eq.3.68} that
\begin{equation}\label{eq.3.69}
	0\equiv p^{rx}-1\equiv (-q^s)^x-1\equiv -(q^{sx}+1) \pmod{q^t+1}.
\end{equation}
Hence, by Lemma \ref{lem.2.3}, we see from \eqref{eq.3.69} that $x$ satisfies \eqref{eq.3.29} and \eqref{eq.3.30}. Further, by \eqref{eq.3.19} and \eqref{eq.3.30}, $x$ satisfies \eqref{eq.3.31}. Therefore, by \eqref{eq.3.23} and \eqref{eq.3.31}, we get from the third equality of \eqref{eq.3.61} that
\begin{equation}\label{eq.3.70}
	p^{rx/y}q^{t/y}\in\mathbb{N},\,\,(p^r+q^s)^z-(p^{rx/y}q^{t/y})^y=1.
\end{equation}
Since $z\ge 3$ and $p^r+q^s>3$, by Lemma \ref{lem.2.2}, we find from \eqref{eq.3.70} that $y$ satisfies \eqref{eq.3.33}. Furthermore, by \eqref{eq.3.20} and \eqref{eq.3.33}, $s$ satisfies \eqref{eq.3.34}. However, the combination of \eqref{eq.3.34} and \eqref{eq.3.66} yields $s\ge z-1\ge q^s$, a contradiction. Thus, this subcase can be eliminated.

To sum up, the theorem is proved.

\section{\bf Acknowledgments} 
We thank an anonymous referee for carefully reading our paper and for giving such constructive comments which substantially helped improving the quality of the paper.

\bigskip 
\hrule 
\bigskip 
\noindent {\it 2020 Mathematics Subject Classification:} Primary 11D61; Secondary 11B39.

\noindent \emph{Keywords:} Lucas sequence, ternary purely exponential Diophantine equation, Nagell-Ljunggren equation, Catalan equation, prime power base, elementary number theory method.

\bigskip
\hrule
\bigskip

\end{document}